\documentclass[11pt, a4paper]{article}
\usepackage[utf8]{inputenc}
\usepackage{amsmath, amsthm, amssymb, amsfonts}
\usepackage{geometry}
\usepackage{graphicx}
\usepackage{hyperref}
\usepackage{enumerate}
\usepackage{booktabs}
\usepackage[numbers]{natbib}
\usepackage{enumitem}

\theoremstyle{plain}
\newtheorem{thm}{Theorem}[section]
\newtheorem{lem}[thm]{Lemma}
\newtheorem{cor}[thm]{Corollary}
\newtheorem{prop}[thm]{Proposition}

\newtheorem{que}[thm]{Question}

\theoremstyle{definition}
\newtheorem{defn}[thm]{Definition}
\newtheorem{exam}[thm]{Example}

\begin{document}

% 标题页
\title{On the $\ell$-th largest degree of an intersecting family}
\author{
  Hao Huang \thanks{Department of Mathematics, National University of Singapore. Email: huanghao@nus.edu.sg. Research supported in part by a start-up grant at NUS and an MOE Academic Research Fund (AcRF) Tier 1 grant A-8003627.}
\and Rui Rao\thanks{Department of Mathematics, National University of Singapore. Email: raorui@u.nus.edu.}
    % 作者二\thanks{电子邮件：author2@institute.edu，机构名称，地址}
}
\date{}
\maketitle

% 摘要
\begin{abstract}
Let $\mathcal{F}\subset\binom{[n]}{k}$ be an intersecting family. For an element $i\in[n]$, the degree of $i$ is the number of sets in $\mathcal{F}$ that contain $i$. Assume that the degrees are ordered as $d_{1}\ge d_{2}\ge\cdots\ge d_{n}$.
Huang and Zhao showed that if $n>2k$, then the minimum degree satisfies $d_{n}\le\binom{n-2}{k-2}$, with the maximum attained by the $1$-star. We strengthen this result by proving that for $n\ge 2k+1$, the $(2k+1)$-th largest degree satisfies $d_{2k+1}\le\binom{n-2}{k-2}$, thereby confirming a conjecture of Frankl and Wang. 
Furthermore, we prove that for large $k$ and $n>12k$, the $(k+2)$-th largest degree  $d_{k+2}$ is already at most $\binom{n-2}{k-2}$. The techniques we developed also yield a tight upper bound for the $(\ell+1)$-th largest degree $d_{\ell+1}$ for $\varepsilon k \le \ell \le k$ and sufficiently large $n>C_{\varepsilon} k$.

\ 

\end{abstract}

% 数学主题分类
% \subjclass[2020]{Primary 00X00; Secondary 00Y00}

% 致谢（可选）
% \acknowledgements{
%     这里可以感谢基金支持、同事讨论等。例如：
%     本项目得到XX基金（编号XXX）支持。感谢XX教授的有益讨论。
% }

% 自动生成目录
% \tableofcontents

% 正文开始
\section{Introduction}
\label{sec:introduction}
We denote by $[n]$ the standard $n$-element set $\{1, 2, \dots, n\}$, $\binom{[n]}{k}$ the family of all $k$-element subsets of $[n]$, and $2^{[n]}$ by the power set of $[n]$. For a family $\mathcal{F} \subset 2^{[n]}$, we say $\mathcal{F}$ is $t$-intersecting if for any two sets $A, B \in \mathcal{F}$, we have $|A \cap B| \ge t$. If $\mathcal{F}$ is $1$-intersecting, we simply call it an intersecting family. If a $t$-intersecting family $\mathcal{F}$ is contained in $\binom{[n]}{k}$, we call $\mathcal{F}$ a $k$-uniform $t$-intersecting family.
%----------------------
% Let $[n]$ denote the standard n-element set $\{1, 2\dots n\}$ and let $\Comb{[n]}{k}$ be the family of $k$ subsets in $[n]$, $2^{[n]}$ be the power set of $[n]$. For a family $\mathcal{F}\subset 2^{[n]}$, we call $\mathcal{F}$ t-intersecting family if for arbitrary two set $A\in \mathcal{F}, B\in\mathcal{F}$, $|A\cap B|\ge t$, if $\mathcal{F}$ is 1-intersecting, then we just call it intersecting set family. If a t-intersecting family $\mathcal{F}\subset \Comb{[n]}{k}$, we called such $\mathcal{F}$ k-uniform t-intersecting set family.
%----------------------
The famous Erd\H{o}s--Ko--Rado theorem \cite{ErdosKoRado} states that if $\mathcal{F}$ is a $k$-uniform intersecting family, then for $n \ge 2k$, $|\mathcal{F}|\le \binom{n-1}{k-1}$.
%------------------------------------------------------
% Erd\H{o}s-Ko-Rado \cite{todo} proves the famous theorem that if $\mathcal{F}$ is k-uniform t-intersecting, then when $n > (t+1)(k-t+1)$, we can bound the size of $\mathcal{F}$ by:
%---------------------------------------------------------
Moreover, equality holds when $\mathcal{F}$ consists of all $k$-element sets containing a fixed element $i$, i.e. when $\mathcal{F}$ is isomorphic to a $1$-star. For this family, the element $i$ appears in $\binom{n-1}{k-1}$ members of $\mathcal{F}$, while every other element appears for $\binom{n-2}{k-2}$ times. 

For a given family $\mathcal{F}\subset \binom{[n]}{k}$ and $i \in [n]$, denote by $d_i(\mathcal{F})$ the degree of $i$, defined as the number of sets in $\mathcal{F}$ that contain $i$, in other word, $d_i(\mathcal{F}) := |\{A|\,i\in A\in \mathcal{F}\}|.$
%-------------------------------------
% the equation can be achieved when $\mathcal{F}$ contains all size $k$ sets containing $[t]$. We can also consider the degree of vertex: for $i\in [n]$, let $d_i(\mathcal{F})$ denote the degree of $i$, defined by the times $i$ appear in $\mathcal{F}$:
%-------------------------------------------------
When the family $\mathcal{F}$ is clear from context and no confusion arises, we simply write $d_i(\mathcal{F})$ as $d_i$. Without loss of generality, we may assume that the degrees are ordered as $d_1 \ge d_2 \ge d_3 \ge \dots \ge d_n$.
%--------------------------------------------
% When the $\mathcal{F}$ is unique and won't cause misunderstand, we will simply write $d_i(\mathcal{F})$ as $d_i$. Without loss of generality, we can assume that $d_1\ge d_2\ge d_3\dots \ge d_n$.
%------------------------------------------------------------$
In 2017, Huang and Zhao \cite{HuangZhao} used spectral graph theory to prove that in a $k$-uniform intersecting set family $\mathcal{F}$, the minimal degree $d_n(\mathcal{F})$ is at most $\binom{n-2}{k-2}$.

\begin{thm}[Huang, Zhao \cite{HuangZhao}]\label{thm:HuangZhao}
Suppose $n > 2k$ and $\mathcal{F} \subset \binom{[n]}{k}$ is an intersecting family. Then $d_n(\mathcal{F}) \le \binom{n-2}{k-2}$.
\end{thm}
\medskip
Recently, Frankl and Wang \cite{FranklWang} observed that in a $k$-uniform intersecting family, the upper bound $\binom{n-2}{k-2}$ in fact governs not only the minimum degree but also most of the degrees. They proved that for $n > 6k-9$, the inequality $d_{2k+1}(\mathcal{F}) \le \binom{n-2}{k-2}$ holds. Moreover, they conjectured that the same bound should remain valid for all $n > 2k$. In other words, every $k$-uniform intersecting family $\mathcal{F} \subset \binom{[n]}{k}$ satisfies $d_{2k+1}(\mathcal{F}) \le \binom{n-2}{k-2}$. If true, this conjecture would substantially strengthen the result of Huang and Zhao (Theorem \ref{thm:HuangZhao}).

In this paper, we verify their conjecture.
% Huang and Zhao using spectral methods to prove that $d_n$ must be small: 
% \begin{thm}[Huang, Zhao, \cite{todo}]\label{HuangZhao}
%     Suppose $n > 2k$, $\mathcal{F}$ is a k-uniform intersecting set family in $\Comb{[n]}{k}$, then $d_n\le \Comb{n-2}{k-2}$
% \end{thm}
% The equation is reached when $\mathcal{F}$ contains all size $k$ set contain 1. Recently, Frankl and Wang sharpen this results when $n>6k-9$:
% They conjectured that $d_{2k+1}(\mathcal{F})\le \Comb{n-2}{k-2}$ for $n\ge 2k+1$, which is proved to be true in this paper.

\begin{thm}\label{thm:2k+1 degree}
Let $n \ge 2k+1$ and let $\mathcal{F} \subset \binom{[n]}{k}$ be an intersecting family. Then
% Suppose $n \ge 2k+1$, $\mathcal{F}$ is a k-uniform intersecting set family, then
    \[d_{2k+1}(\mathcal{F})\le \binom{n-2}{k-2}.\]
\end{thm}

A natural question is whether we can find more degrees that are bounded by $\binom{n-2}{k-2}$. The following extremal example of the Hilton--Milner theorem \cite{HiltonMilner} provides a construction demonstrating that $d_{k+1}(\mathcal{F}) \ge \binom{n-2}{k-2} + 1$ may occur even for sufficiently large $n$. 
\begin{exam}[Hilton--Milner]\label{exam:Hilton Milner}
    Let $\mathcal{F} \subset \binom{[n]}{k}$ be the $k$-uniform intersecting set family consisting of the following subsets:
    \begin{enumerate}[label=(\alph*), noitemsep, topsep=2pt]
        \item $\{2, 3, \dots, k+1\}$;\smallskip
        \item all $k$-element sets that contain $1$ and intersect the set $\{2, 3, \dots, k+1\}$.
    \end{enumerate}
It is not hard to check that $d_1(\mathcal{F})=\binom{n-1}{k-1}-\binom{n-k-1}{k-1}$, and $d_2(\mathcal{F})=\cdots=d_{k+1}(\mathcal{F})=\binom{n-2}{k-2}+1$.
\end{exam}
Frankl and Wang \cite{FranklWang} showed that for a $k$-uniform $t$-intersecting family $\mathcal{F}$, if $n > \binom{t+2}{2}k^2$, then $d_{k+2}(\mathcal{F}) \le \binom{n-t-1}{k-t-1}$. They also conjectured that the quadratic condition $n > \binom{t+2}{2}k^2$ can be strengthened to a linear range: there exists an absolute constant $c$ such that whenever $n > c k t$, any $k$-uniform $t$-intersecting family $\mathcal{F}$ satisfies $d_{k+2}(\mathcal{F}) \le \binom{n-t-1}{k-t-1}$. We verify their conjecture for $t = 1$.

% Frankl and Wang also prove that when $n$ is large, $d_{k+2}$ should be small:
% \begin{thm}[Frankl-Wang, \cite{todo}]\label{Baseline}
%     Suppose $n > \Comb{t+2}{2}k^2$, and $\mathcal{F}$ is a k-uniform t-intersecting set family, then $d_{k+2}\le \Comb{n-t-1}{k-t-1}$
% \end{thm}
% They also made the following stronger conjecture:

\begin{thm}\label{thm:k+2 degree}
    There exist an absolute constant $C = 12$ such that, for sufficiently large $k$ and for $n > Ck$, let $\mathcal{F} \subset \binom{[n]}{k}$ be an intersecting family. Then
    \[d_{k+2}(\mathcal{F}) \le \binom{n-2}{k-2}.\]
\end{thm}

Next we consider the $\ell$-th largest degree of an intersecting family, Frankl and Wang \cite{FranklWang} determined the maximum possible values for $d_{\ell}(\mathcal{F})$ for $2 \le \ell \le 4$. For the $(\ell+1)$-degree problem when $\ell\ge 4$, inspired by the Hilton--Milner construction, they provided the following example.
\begin{exam}\label{exam:extremal}
    Let $\mathcal{F}$ be the $k$-uniform intersecting family consisting of all the following sets:

    \begin{enumerate}[label=(\alph*), noitemsep, topsep=2pt]
        \item all $k$-element sets containing the element $1$ and at least one element from $\{2, 3, \dots, \ell+1\}$;\smallskip
        \item all $k$-element sets containing the set $\{2, 3, \dots, \ell+1\}$.
    \end{enumerate}
    Then $d_{\ell+1}(\mathcal{F}) = \binom{n-2}{k-2} + \binom{n-\ell-1}{k-\ell}$.
\end{exam}
They also proved that for $4 \le \ell \le k$, $d_{\ell+1}\le \binom{n-2}{k-2} + \binom{n-\ell-1}{k-\ell}$ when $n>2\ell^2k$. In this paper, we obtain an improved range for large $\ell$ via the following theorem.
% We verify that when $t = 1$, this conjecture is true.
% \begin{thm}\label{k+2 degree}
%     Assume $k > 50$, $n > \dfrac{11}{2}k$, $\mathcal{F}$ is a k-uniform intersecting set family, then
%     \[d_{k+2}(\mathcal{F})\le \Comb{n-2}{k-2}\]
% \end{thm}
% Hilton-Milner theorem give an example that $d_{k+1}\ge \Comb{n-2}{k-2}+1$ for arbitrary large $n$, so $k+2$ is the  best possible we can get. The following theorem shows that when $n$ is large enough, Hilton-Milner construction maximize $d_{k+1}$
% \begin{thm}[Frankl, Wang, \cite{todo}]\label{FW l+1 degree}
% Suppose $\mathcal{F}$ is a k-uniform intersecting set family, $4\le l\le k$ and $n>2l^2k$, then
% \[d_{l+1}\le \Comb{n-2}{k-2} + \Comb{n-l-1}{k-l}\]
% \end{thm}

\begin{thm}\label{thm:l+1 degree}
    For any constant $\varepsilon$, there is a constant $C_\varepsilon$ such that, for sufficiently large $k$ and for $\varepsilon k\le \ell \le k+1 $, if $n>C_\varepsilon k$ and $\mathcal{F}\subset \binom{[n]}{k}$ is an intersecting family, then
    \[d_{\ell+1}(\mathcal{F}) \le \binom{n-2}{k-2} + \binom{n-\ell-1}{k-\ell}.\]
\end{thm}
% We strengthen their results and get the following theorem:
% \begin{thm}\label{l+1 degree}
%    Suppose $\mathcal{F}$ is a k-uniform intersecting set family, $20\le l\le k$ and $n>\dfrac{k^2}{l}+32k$, then
% \[d_{l+1}\le \Comb{n-2}{k-2} + \Comb{n-l-1}{k-l}\] 
% \end{thm}

% \begin{exam}
%     Let $\mathcal{F}$ be the k-uniform intersecting set family contains all following sets:
    
%     \ \ a) all sets contains $1$ and at least one vertex in $\{2, 3\dots l+1\}$

%     \ \ b) all sets contains $\{2, 3\dots l+1\}$

%     Then for such $\mathcal{F}$, 
%     \[d_{l+1}(\mathcal{F}) = \Comb{n-2}{k-2} + \Comb{n-l-1}{k-l}\]
%     In this example the inequality in Theorem \ref{FW l+1 degree} and Theorem \ref{l+1 degree} takes equation.
% \end{exam}

The remainder of the paper is organized as follows. In Section \ref{sec:Shifting Method}, we review the shifting methods that we will apply, and present several related lemmas. Section \ref{sec:2k+1 degree} contains the proof of Theorem \ref{thm:2k+1 degree}. In Section \ref{sec:l+1 degree}, we employ similar techniques to establish both Theorem \ref{thm:k+2 degree} and Theorem \ref{thm:l+1 degree}. Finally, Section \ref{sec:conclusion} concludes with a summary of our results and two open questions.

% 引用文献采用上标形式，如：近年来相关研究取得了显著进展\cite{key1, key2}。

\section{The shifting method}\label{sec:Shifting Method}

In this section we give a brief description of the shifting method. This technique is particularly effective when dealing with intersecting families and will be employed in the proofs of all of our three theorems.

\begin{defn}\label{defn:shift set}
    For a given family $\mathcal{F}$ of subsets and $i, j \in [n]$ with $i \neq j$, the \emph{$ij$-shifting} of a set $A \in \mathcal{F}$ is defined as:
    \[
    A_{ij} = 
    \begin{cases}
        (A \cup \{j\}) \setminus \{i\}, & \text{if } i \in A,\ j \notin A,\ (A \cup \{j\}) \setminus \{i\} \notin \mathcal{F}; \\
        A, & \text{otherwise}.
    \end{cases}
    \]
    The \emph{$ij$-shifting} of the family $\mathcal{F}$ is defined by $\mathcal{F}_{ij} := \{A_{ij} \mid A \in \mathcal{F}\}$.
\end{defn}
% In this section we give a briefly description to the shifting method on intersecting problem. It can be written as the following lemma:

% \begin{defn}
% Let $A\in\mathcal{F}$ is a subset of $[n]$, let $i\not = j \in [n]$, we defined the $ij$-shifting of $A$ as:

% \[A_{ij} = \begin{cases}
%     (A\cup\{j\})\backslash\{i\}, \text{if $i\in A, j\not\in A, (A\cup\{j\}\backslash\{i\})\not\in \mathcal{F}$}\\
%     A, \text{otherwise}
% \end{cases}\]
% \end{defn}

% \begin{defn}
%     The $ij$ shifting of $\mathcal{F}$ is defined as followed: $\mathcal{F}_{ij} : = \{A_{ij}|A\in \mathcal{F}\}$
% \end{defn}

\begin{lem}[Shifting lemma for intersecting family, \cite{DezaFrankl}]\label{lem:shifting}
    If $\mathcal{F}$ is a $t$-intersecting family, then $\mathcal{F}_{ij}$ is also a $t$-intersecting family.
\end{lem}

The shifting lemma also holds for cross-intersecting families. We say two families $\mathcal{F}_1$ and $\mathcal{F}_2$ are \emph{$t$-cross-intersecting} if $|A \cap B| \ge t$ for all $A \in \mathcal{F}_1$ and $B \in \mathcal{F}_2$. We then have the following lemma:

\begin{lem}[Shifting lemma for cross-intersecting families, \cite{Frankl2014}]\label{lem:cross shifting}
    If two families $\mathcal{F}_1$ and $\mathcal{F}_2$ are $t$-cross-intersecting, then $(\mathcal{F}_1)_{ij}$ and $(\mathcal{F}_2)_{ij}$ are also $t$-cross-intersecting.
\end{lem}
% \begin{lem}[shifting lemma, \cite{todo}]\label{shifting}
% Assume $\mathcal{F}$ is t-intersecting set family, then $\mathcal{F}_{ij}$ is also t-intersecting set family
% \end{lem}

% The shifting lemma is also holds for cross-intersecting family, we say two set family $\mathcal{F}_1$ and $\mathcal{F}_2$ are t-cross-intersecting, if for arbitrary $A\in \mathcal{F}_1, B\in \mathcal{F}_2, |A\cap B|\ge t$, then we have the following lemma:

% \begin{lem}[shifting lemma, \cite{todo}]\label{cross-shifting}
%     Assume $\mathcal{F}_1$ and $\mathcal{F}_2$ are two t-cross-intersecting set families, then $(\mathcal{F}_1)_{ij}$ and $(\mathcal{F}_2)_{ij}$ are also t-cross-intersecting set families.
% \end{lem}

\begin{defn}[Lexicographic order]
    The \emph{lexicographic order} is a total order $<_L$ on $\binom{[n]}{k}$ defined by:
    \[
    A <_L B \quad \text{if and only if} \quad \min(A \setminus B) < \min(B \setminus A).
    \]
    
    The \emph{co-lexicographic order} is a total order $<_C$ on $\binom{[n]}{k}$ defined by:
    \[
    A <_C B \quad \text{if and only if} \quad \max(A \setminus B) < \max(B \setminus A).
    \]
\end{defn}

\begin{defn}[Shadow]
    Let $\mathcal{F} \subset \binom{[n]}{k}$ be a $k$-uniform family. The \emph{shadow} of $\mathcal{F}$, denoted by $\partial\mathcal{F}$, is the $(k-1)$-uniform family defined by:
    \[
    \partial\mathcal{F} := \{ A \mid A = B \setminus \{x\} \text{ for some } x \in B \in \mathcal{F} \}.
    \]
\end{defn}

\begin{thm}[Kruskal–Katona, \cite{Hilton,Katona,Kruskal}]\label{thm:shadow KK}
    Let $\mathcal{F} \subset \binom{[n]}{k}$ with $|\mathcal{F}| = m$, and let $\mathcal{C}(k, m)$ denote the family of the first $m$ $k$-sets in the co-lexicographic order. Then the following inequality holds:
    \[
    |\partial\mathcal{F}| \ge |\partial\mathcal{C}(k, m)|.
    \]
\end{thm}
% \begin{cor}\label{cross-intersecting degree}
%     Under the same assumption of Theorem \ref{cross-intersecting KK}, we have either $x_1\le \Comb{n-t}{a-t}$ or $x_2\le \Comb{n-t}{b-t}$
% \end{cor}

% \begin{proof}
%     The first $\Comb{n-t}{a-t}$ size $a$ sets in the lexicographic order are those $\Comb{n-t}{a-t}$ size $a$ sets contains $[t]$, similar the first $\Comb{n-t}{b-t}$ size $b$ sets in the lexicographic order are those $\Comb{n-2}{b-t}$ sets contain $[t]$.
% \end{proof}

% \begin{defn}[shadow]
%     Assume $\mathcal{F}$ is a $k$-uniform set family of $[n]$, the shadow of $\mathcal{F}$, denote by $\partial\mathcal{F}$, is the $k-1$-uniform set family defined by:
%     \[\partial\mathcal{F}:=\{A|A = B\backslash \{x\}\text{ for some }x\in B\in \mathcal{F}\}\]
% \end{defn}

% \begin{thm}[Kruskal-Katona, ?]\label{shadow KK}
%     Assume $\mathcal{F}$ is a k-uniform set family, $|\mathcal{F}| = m$. Let $\mathcal{C}(k, m)$ denote the first $m$ k-size set under the co-lexicographic order. Then the following inequality holds:
%     \[|\partial\mathcal{F}|\ge |\partial\mathcal{C}(k, m)|\]
% \end{thm}
\begin{cor}\label{cor:up shadow KK}
    Let $\overline{\partial}\mathcal{F}$ denote the family
    \[
    \overline{\partial}\mathcal{F} = \{ A \mid A = B \cup \{x\} \text{ for some } B \in \mathcal{F},\ x \in [n] \setminus B \}.
    \]
    Suppose $\mathcal{F} \subset \binom{[n]}{k}$ and $|\mathcal{F}| > \binom{n-1}{k-1}$. Then $|\overline{\partial}\mathcal{F}| > \binom{n-1}{k}$.
\end{cor}

\begin{proof}
    Define $\mathcal{F}^C := \{[n] \setminus A \mid A \in \mathcal{F}\}$. Then $|\mathcal{F}^C| = |\mathcal{F}|$, and
    \[
    \overline{\partial}\mathcal{F} = (\partial\mathcal{F}^C)^C.
    \]
    Hence it suffices to prove that if $|\mathcal{F}^C| > \binom{n-1}{k-1}$, then $|\partial\mathcal{F}^C| > \binom{n-1}{k}$.
    
    By Theorem \ref{thm:shadow KK}, we may assume that $\mathcal{F}^C$ contains the first $\binom{n-1}{k-1} + 1$ $(n-k)$-sets in the co-lexicographic order. These are all the $(n-k)$-sets contained in $[n-1]$, together with one additional set containing the element $n$.
    Consequently, $\partial\mathcal{F}^C$ must contain all $(n-k-1)$-sets contained in $[n-1]$, plus at least one set containing $n$. Thus,
    \[
    |\partial\mathcal{F}^C| \ge \binom{n-1}{n-k-1} + 1 = \binom{n-1}{k} + 1.
    \]
\end{proof}
% \begin{cor}\label{up shadow KK}
%     let $\overline{\partial}\mathcal{F}$ denote the following sets:
% \[\overline{\partial}\mathcal{F} = \{A|A = B\cup\{x\}\text{ for some $B\in\mathcal{F}, x\in [n]\backslash B$}\}\]
% Suppose $\mathcal{F}$ is a k-uniform set family and $|\mathcal{F}|> \Comb{n-1}{k-1}$, then $|\overline{\partial}\mathcal{F}|> \Comb{n-1}{k}$ 
% \end{cor}
% \begin{proof}
%     let $\mathcal{F}^C: = \{[n]\backslash A|A\in \mathcal{F}\}$, then $|\mathcal{F}^C| = |\mathcal{F}|$, and 
%     \[\overline{\partial}\mathcal{F} = (
%     \partial\mathcal{F}^C)^C
%     \]

%     So it's sufficient to prove when $|\mathcal{F}^C|> \Comb{n-1}{k-1}$, $|\partial\mathcal{F}^C|> \Comb{n-1}{k}$

%     By the Theorem \ref{shadow KK}, we can assume that $\mathcal{F}^C$ contains the first $\Comb{n-1}{k-1} + 1$ $(n-k)$-size set in the co-lexicographic order, which is all the $n-k$ sets in $[n-1]$ and one set contain $n$.

%     Therefore $\partial\mathcal{F}^C$ must contains all size $n-k-1$ sets in $[n-1]$ and at least one set contain $n$.
%     \[|\partial\mathcal{F}^C|\ge \Comb{n-1}{n-k-1} + 1 = \Comb{n-1}{k} + 1\]
% \end{proof}

\begin{defn}[$\ell$-shifted family]
    A $k$-uniform family $\mathcal{F} \subset \binom{[n]}{k}$ is called \emph{$\ell$-shifted} if $\mathcal{F}_{ij} = \mathcal{F}$ for every pair $(i,j)$ with $i \notin [\ell]$ and $j \in [\ell]$.
\end{defn}

\begin{prop}\label{prop:shifted assumption}
    Let $\mathcal{F}\subset \binom{[n]}{k}$ be a $t$-intersecting family and the degrees are ordered as $d_1(\mathcal{F})\ge d_2(\mathcal{F})\dots \ge d_n(\mathcal{F})$. Then there exists an $\ell$-shifted $t$-intersecting $k$-uniform family $\mathcal{F}'$ such that $d_i(\mathcal{F}') \ge d_i(\mathcal{F})$ for $i\le \ell$. Furthermore, elements $[\ell]$ are still the $\ell$ elements with the largest degrees in $\mathcal{F'}$. Consequently, the $\ell$-largest degree of $\mathcal{F}'$ is not smaller than $d_\ell(\mathcal{F})$. 
\end{prop}

\begin{proof}
    If $\mathcal{F}$ is not $\ell$-shifted, then there exist indices $i \notin [\ell]$ and $j \in [\ell]$ such that $\mathcal{F}_{ij} \neq \mathcal{F}$. Replace $\mathcal{F}$ with $\mathcal{F}_{ij}$. By Lemma \ref{lem:shifting}, $\mathcal{F}_{ij}$ remains $t$-intersecting. Note that the $ij$-shifting will not decrease the degree of $k\not= i$, so it won't decrease the degree of $k\in[\ell]$. Therefore elements of $[\ell]$ give the $\ell$ largest degrees of $\mathcal{F'}$, $d_i(\mathcal{F}_{ij})\ge d_i(\mathcal{F})\ge  d_{\ell}(\mathcal{F})$ for $i\le \ell$. We may repeatedly apply the $ij$-shifting for different pairs $(i, j)$ until this process terminates after finitely many steps, yielding the desired family $\mathcal{F}'$.
\end{proof}
% One fact about the shifting method is that, the $ij$-shifting of $\mathcal{F}$ will only decrease the degree $d_j$, hence we can define a $l$-shifted set family as following:
% \begin{defn}[$l$-shifted set family]
%     A $k$-uniform set family $\mathcal{F}$ is called $l$-shifted, if for all $i\not\in l, j\in l$, $\mathcal{F}_{ij} = \mathcal{F}$.
% \end{defn}

% \begin{prop}\label{shifted assumption}
%     Suppose $\mathcal{F}$ is a $t$-intersecting $k$-uniform set family, then there exist a $l$-shifted k-uniform $t$-intersecting set family $\mathcal{F}'$ such that $d_l(\mathcal{F}')\ge d_l(\mathcal{F})$.
% \end{prop}
% \begin{proof}
%    If $\mathcal{F}$ is not $l$-shifted, says $\mathcal{F}_{ij}\not = \mathcal{F}$ for $i>l, j\le l$, then replace $\mathcal{F}$ by $\mathcal{F}_{ij}$, this operation ensure that $d_l(\mathcal{F}_{ij})\ge d_l(\mathcal{F})$, $\mathcal{F}_{ij}$ is still $t$-intersecting by Lemma \ref{shifting}, and the shifting will end in finite step, then we get the $\mathcal{F}'$ we want.
% \end{proof}
In the proof of our theorems, Proposition \ref{prop:shifted assumption} allows us to assume that $\mathcal{F}$ is $\ell$-shifted. Furthermore, we can re-name the elements in the ground set so that the degrees are still in non-increasing order. The resulting family might not be $\ell$-shifted, but then we can apply shiftings again. Since at least one of the $\ell$ largest degrees strictly increases, while there are only finitely many subsets in the family, this process will eventually terminate, and we end up with a $\ell$-shifted family $\mathcal{F}$ whose degrees are in non-increasing order.

\begin{prop}\label{prop:set shifting}
    Let $\mathcal{F}$ be an $\ell$-shifted $t$-intersecting $k$-uniform family. For $A \in \mathcal{F}$, suppose there are elements $j \le \ell$, $i > \ell$ such that $i \in A$ and $j \notin A$. Then $(A \cup \{j\}) \setminus \{i\} \in \mathcal{F}$.
\end{prop}

\begin{proof}
    Since $\mathcal{F}_{ij} = \mathcal{F}$ by the definition of $\ell$-shifted family, we have $A_{ij} \in \mathcal{F}_{ij} = \mathcal{F}$. Thus $(A \cup \{j\}) \setminus \{i\} = A_{ij}\in \mathcal{F}$.
\end{proof}

Applying Proposition \ref{prop:set shifting}, we obtain the following property.
% \begin{prop}\label{set shifting}
%     Suppose $\mathcal{F}$ is a $l$-shifted $t$-intersecting $k$-uniform set family, $A\in \mathcal{F}, j\le l, i> l$ satisfy $i\in A, j\not\in A$, then $(A\cup\{j\})\backslash\{i\}\in \mathcal{F}$.
% \end{prop}
% \begin{proof}
%     Because $\mathcal{F}_{ij} = \mathcal{F}$, so if $A_{ij}\not = A$, then $A_{ij}\in \mathcal{F}$.
% \end{proof}

% Applying Proposition \ref{set shifting}, we can derive the following property:

\begin{lem}\label{lem:intersecting lemma}
    Let $\mathcal{F}$ be an $\ell$-shifted intersecting $k$-uniform family. Define the (not necessarily uniform) family $\mathcal{G}$ on $[\ell]$ by:
    \[
    \mathcal{G} = \{ A \cap [\ell] \mid A \in \mathcal{F} \}.
    \]
    Then the following statements hold:
    \begin{enumerate}[label=(\alph*)]
        \item $\mathcal{G}$ is upward closed among subsets of $[\ell]$ of size at most $k$, meaning that for all $A\subset B\subset [\ell]$, if $A\in\mathcal{G}$ and $|B|\le k$, then $B\in \mathcal{G}$.
        \item If $\ell \ge 2k$, then $\mathcal{G}$ is intersecting.
        \item If $A, B \in \mathcal{F}$ satisfy $A \cap B \cap [\ell] = \emptyset$, then
        \[
        |A \cap [\ell]| + |B \cap [\ell]| + |A \cap B| \ge \ell + 1.
        \]
    \end{enumerate}
\end{lem}

\begin{proof}
    (a) Suppose $A \in \mathcal{G}$ and $|A| < k$. Then there exists $A_0 \in \mathcal{F}$ such that $A_0 \cap [\ell] = A$. For any $j \in [\ell] \setminus A$, choose an element $i \in A_0 \setminus A$ (note that $i \notin [\ell]$ since $A_0 \cap [\ell] = A$). By Proposition \ref{prop:set shifting},
    \[
    A_1 := (A_0 \cup \{j\}) \setminus \{i\} \in \mathcal{F}.
    \]
    Then $A_1 \cap [\ell] = (A \cup \{j\}) \cap [\ell] = A \cup \{j\}$. Hence $A \cup \{j\} \in \mathcal{G}$ for arbitrary $j \in [\ell]\setminus A$, proving that $\mathcal{G}$ is upward closed for sets of size at most $k$.\medskip

    (b) Suppose for contradiction that $A, B \in \mathcal{F}$ with $A \cap B \cap [\ell] = \emptyset$. Then we know that $B\cap [\ell]\subset [\ell]\setminus A$. Choose a set $C$ satisfy $B\cap [\ell]\subset C\subset [\ell] \setminus A$ with $|C| = k$. By repeatedly applying Proposition \ref{prop:set shifting}, we obtain that $C\in \mathcal{F}$, but $C\cap A = \emptyset$, a contradiction.\medskip

    (c) Assume $A, B \in \mathcal{F}$ with $A \cap B \cap [\ell] = \emptyset$. Let $C = A \cap B$ and suppose, for contradiction, that
    \[
    |C| \le \ell - |A \cap [\ell]| - |B \cap [\ell]|.
    \]
    Then $|[\ell] \setminus (A \cup B)| = \ell - |A \cap [\ell]| - |B \cap [\ell]| \ge |C|$. Choose a set $D \subseteq [\ell] \setminus (A \cup B)$ with $|D| = |C|$. By repeatedly applying Proposition \ref{prop:set shifting} to replace each element of $C$ in $A$ with a distinct element of $D$, we obtain that
    \[
    A' := (A \cup D) \setminus C \in \mathcal{F}.
    \]
    However, $A' \cap B = \emptyset$, contradicting the assumption that $\mathcal{F}$ is intersecting.
\end{proof}

\section{The $(2k+1)$-th largest degree is at most $\binom{n-2}{k-2}$}\label{sec:2k+1 degree}

In this section we prove Theorem \ref{thm:2k+1 degree}, which states that for any intersecting family $\mathcal{F} \subset \binom{[n]}{k}$ with $n > 2k$, its $(2k+1)$-th largest degree satisfies $d_{2k+1}(\mathcal{F}) \le \binom{n-2}{k-2}$.

Proposition \ref{prop:shifted assumption} allows us to assume that $\mathcal{F}$ is $(2k+1)$-shifted. Define the (not necessarily uniform) family $\mathcal{G}$ on $[2k+1]$ by:
\[
\mathcal{G} := \{ A \cap [2k+1] \mid A \in \mathcal{F} \}.
\]
Then part (b) of Lemma \ref{lem:intersecting lemma} guarantees that $\mathcal{G}$ is intersecting. Our general strategy is to find an element $i \in [2k+1]$ such that, for each $j \le k$, $i$ appears in relatively few $j$-element sets of $\mathcal{G}$.

\begin{defn}[$m$-degree]
    Let $\mathcal{G}$ be a (not necessarily uniform) family of subsets of $[n]$. For $i \in [n]$, the \emph{$m$-degree} of $i$ in $\mathcal{G}$, denoted by $d_i^{(m)}(\mathcal{G})$, is the number of sets in $\mathcal{G}$ of size $m$ that contain $i$:
    \[
    d_i^{(m)}(\mathcal{G}) := \bigl| \{ A \in \mathcal{G} \mid i \in A,\ |A| = m \} \bigr|.
    \]
    We shall simply write $d_i^{(m)}$ when the family $\mathcal{G}$ is clear from the context.
\end{defn}
% Proposition \ref{prop:shifted assumption} suggest us to focus on the $2k+1$-shifted intersecting $k$-uniform set family. Let $\mathcal{G}$ denote the following (not uniform) set family:
% \[\mathcal{G}:=\{A\cap [2k+1]|A\in \mathcal{F}\}\]
% Then Lemma \ref{intersecting lemma} b) ensure that $\mathcal{G}$ is intersecting, our general idea is proving that there is an element with low degree in $\mathcal{G}$.

% \begin{defn}[m-degree]
% Let $\mathcal{G}$ to be a (not-uniform) set family of $[n]$, for $i\in [n]$, the m-degree of $i$, denoted by $d_i^{(m)}(\mathcal{G})$, is the number of $m$-size sets contain $i$ in $\mathcal{G}$:
% \[d_i^{(m)}(\mathcal{G}) := \big|\{A\big|i\in A\in\mathcal{G}, |A| = m\}\big|\]
% We will simply write $d_i^{(m)}$ when the set family $\mathcal{G}$ is clear in the context.
% \end{defn}
\begin{lem}\label{lem:uniform low degree}
    Let $n > 2k$ and let $\mathcal{G}$ be a (not necessarily uniform) intersecting family of subsets of $[n]$. Then there exists an element $i \in [n]$ such that for every $m \le k$,
    \[
    d_i^{(m)}(\mathcal{G}) \le \binom{n-2}{m-2}.
    \]
\end{lem}

\begin{proof}
    Let $\overline{\mathcal{G}}$ denote the upward closure of $\mathcal{G}$:
    \[
    \overline{\mathcal{G}} = \{ A \subseteq [n] \mid A \supseteq B \text{ for some } B \in \mathcal{G} \}.
    \]
    Then $\overline{\mathcal{G}}$ is still an intersecting family. Consequently, the $k$-element sets in $\overline{\mathcal{G}}$ form a $k$-uniform intersecting family on $[n]$ with $n>2k$. By Theorem \ref{thm:HuangZhao}, there exists an element $i \in [n]$ such that $d_i^{(k)}(\overline{\mathcal{G}}) \le \binom{n-2}{k-2}$.

    We now prove by reverse induction on $m$ that
    \[
    d_i^{(m)}(\overline{\mathcal{G}}) \le \binom{n-2}{m-2} \quad \text{for all } m \le k.
    \]
    The base case $m = k$ has already been established. Assume inductively that the statement holds for $m+1$ (where $m+1 \le k$), and suppose for contradiction that
    \[
    d_i^{(m)}(\overline{\mathcal{G}}) > \binom{n-2}{m-2}.
    \]
    Define
    \[
    \overline{\mathcal{G}}_i^{(m)} := \{ A \setminus \{i\} \mid i \in A \in \overline{\mathcal{G}},\ |A| = m \}.
    \]
    This is an $(m-1)$-uniform family on the ground set $[n] \setminus \{i\}$. Because $\overline{\mathcal{G}}$ is upward closed, we can apply Corollary \ref{cor:up shadow KK} to $\overline{\mathcal{G}}_i^{(m)}$ on the ground set $[n]\setminus\{i\}$ and obtain $|\,\overline{\partial}(\overline{\mathcal{G}}_i^{(m)})\,| > \binom{n-2}{m-1}$. Moreover, each $m$-set in $\overline{\partial}(\overline{\mathcal{G}}_i^{(m)})$ corresponds to an $(m+1)$-set in $\overline{\mathcal{G}}$ obtained by adding the element $i$, which yields
    \[
    d_i^{(m+1)}(\overline{\mathcal{G}}) \ge |\,\overline{\partial}(\overline{\mathcal{G}}_i^{(m)})\,| > \binom{n-2}{m-1},
    \]
    contradicting the induction hypothesis for $m+1$. Hence,
    \[
    d_i^{(m)}(\overline{\mathcal{G}}) \le \binom{n-2}{m-2} \quad \text{for all } m \le k.
    \]
    Finally, since every set in $\mathcal{G}$ is contained in  $\overline{\mathcal{G}}$, we have $d_i^{(m)}(\mathcal{G}) \le d_i^{(m)}(\overline{\mathcal{G}})$, which completes the proof.
\end{proof}
% \begin{lem}\label{uniform low degree}
%     Suppose $n>2k$, $\mathcal{G}$ is a (not uniform) intersecting set family of $[n]$, then there is $i\in [n]$, such that for all $m\le k$, $d_i^{(m)}(\mathcal{G})\le \Comb{n-2}{m-2}$.  
% \end{lem}
% \begin{proof}
%     we let $\overline{\mathcal{G}}$ denote the upward closure of $\mathcal{G}$:
%     \[\overline{\mathcal{G}} = \{A|A\supset B\text{ for some $B\in G$}\}\]
%     Then $\overline{\mathcal{G}}$ is still an intersecting set, so all the $k$-size set of $\overline{\mathcal{G}}$ forms a k-uniform intersecting set family of $[n]$. By the Theorem \ref{HuangZhao}, there is some $i\in[n]$ such that $d_i^{(k)}(\overline{\mathcal{G}})\le \Comb{n-2}{k-2}$.

%     Now we induction on $m$ to prove that $d_i^{(m)}(\overline{\mathcal{G}})\le \Comb{n-2}{m-2}$ for all $m\le k$

%     We have proved that when $m = k$, the statement is true, now we assume that the statement is true for $m+1$.

%     Suppose $d_i^{(m)}(\overline{\mathcal{G}})> \Comb{n-2}{m-2}$, let $\overline{\mathcal{G}}_i^{(m)} = \{A\backslash \{i\}\big|i\in A\in \overline{\mathcal{G}}, |A| = m\}$, then $\overline{\mathcal{G}}_i^{(m)}$ is a $m-1$ uniform set family in $[n]\backslash\{i\}$. Because $\overline{\mathcal{G}}$ is upward closed, by the Corollary \ref{up shadow KK},
%     \[d_i^{(m+1)}(\overline{\mathcal{G}})\ge |\overline{\partial}(\overline{\mathcal{G}}_i^{(m)})| > \Comb{n-2}{m-1}\]
%     Contradiction to the induction hypothesis. Therefore,
%     \[d_i^{(m)}(\mathcal{G})\le d_i^{(m)}\overline{\mathcal{G}} = \Comb{n-2}{m-2}\]
%     For all $m\le k$.
% \end{proof}

\noindent\textbf{Proof of Theorem \ref{thm:2k+1 degree}.}
By Proposition \ref{prop:shifted assumption}, we may assume that $\mathcal{F}$ is $(2k+1)$-shifted, and $d_1, d_2, \dots, d_{2k+1}$ are the $(2k+1)$ largest degrees.

Define $\mathcal{G} = \{ A \cap [2k+1] \mid A \in \mathcal{F} \}$. By part (b) of Lemma \ref{lem:intersecting lemma}, $\mathcal{G}$ is an intersecting family on $[2k+1]$. Applying Lemma \ref{lem:uniform low degree} with $n = 2k+1$, we obtain an element $i \in [2k+1]$ such that for every $m \le k$,
\[
d_i^{(m)}(\mathcal{G}) \le \binom{2k-1}{m-2}.
\]

Observe that for a fix set $A \in \mathcal{G}$. There are at most $\binom{n-2k-1}{k - |A|}$ sets $B \in \mathcal{F}$ satisfying $B \cap [2k+1] = A$. Consequently, we can bound the degree of $i$ in $\mathcal{F}$ as follows:
\[
\begin{aligned}
d_i(\mathcal{F}) &\le \sum_{m=1}^{k} d_i^{(m)}(\mathcal{G}) \binom{n-2k-1}{k-m} \\
&\le \sum_{m=1}^{k} \binom{2k-1}{m-2} \binom{n-2k-1}{k-m} \\
&= \binom{n-2}{k-2}.
\end{aligned}
\]

Hence, $d_{2k+1}(\mathcal{F}) \le d_i(\mathcal{F}) \le \binom{n-2}{k-2}$, which completes the proof. \qed

% \noindent\textbf{Proof of theorem \ref{2k+1 degree}:}
% By Proposition \ref{shifted assumption}, we can assume that $\mathcal{F}$ is also $(2k+1)$-shifted.

% Take $\mathcal{G} = \{A\cap[2k+1]|A\in \mathcal{F}\}$, then by Lemma \ref{intersecting lemma}, b), $\mathcal{G}$ is an intersecting set family of $[2k+1]$. 

% By Lemma \ref{uniform low degree}, there is some $i\in [2k+1]$ such that for all $m\le k$, $d_i^{(m)}(\mathcal{G})\le \Comb{2k-1}{m-2}$.

% Fix $A\subset \mathcal{G}$, there are at most $\Comb{n-2k-1}{k-|A|}$ $k$-size set $B\in\mathcal{F}$ such that $B\cap [2k+1] = A$.

% Therefore, we can estimate the degree of $i$ in $\mathcal{F}$ by:
% \[\begin{aligned}
%     d_i(\mathcal{F})\le \sum_{m\le k} d_i^{(m)}(\mathcal{G})\Comb{n-2k-1}{k-m}\le \sum_{m\le k}\Comb{2k-1}{m-2}\Comb{n-2k-1}{k-m} = \Comb{n-2}{k-2}
% \end{aligned}\]

% Therefore $d_{2k+1}(\mathcal{F})\le d_i(\mathcal{F})\le \Comb{n-2}{k-2}$.

\ 

\section{Estimating the $(\ell+1)$-th largest degree} \label{sec:l+1 degree}

In this section, we first prove Theorem \ref{thm:k+2 degree}, which states that for $k$ sufficiently large and $n > 12k$, any intersecting family $\mathcal{F} \subset \binom{[n]}{k}$ satisfies $d_{k+2}(\mathcal{F}) \le \binom{n-2}{k-2}$. As we mentioned earlier, one cannot hope for $d_{k+1}(\mathcal{F}) \le \binom{n-2}{k-2}$ due to the Hilton--Milner example. Therefore to have the $1$-star as the extremal example, the index $k+2$ is best possible. We then apply similar techniques to prove Theorem \ref{thm:l+1 degree}, which determines the maximum value of $d_{\ell+1}(\mathcal{F})$ when $\ell = \Theta(k) $ and $n$ is sufficiently large.

For a fixed $\ell$ and a subset $S\subset[\ell]$, we define the \emph{degree of $S$} (with respect to $\mathcal{F}$) as the size of the family
\[
d_S(\mathcal{F}) = \bigl| \{ A \in \mathcal{F} \mid A \cap [\ell] = S \} \bigr|.
\]
We first prove that if $\mathcal{A}\subset \binom{[n]}{a}$ and $\mathcal{B}\subset\binom{[n]}{b}$ are $t$-cross-intersecting, then their sizes can be bounded as follows.
\begin{lem}\label{lem:cross-intersecting size}
    Assume $n>3\max\{a, b\}$, $\mathcal{A}\subset \binom{[n]}{a}, \mathcal{B}\subset\binom{[n]}{b}$ are non-empty $t$-cross-intersecting families. Then we have $$|\mathcal{A}|\le 2^{b}\binom{n-b}{a-t}, ~~|\mathcal{B}|\le 2^{a}\binom{n-a}{b-t}.$$
\end{lem}
\begin{proof}
    We only prove the inequality for $|\mathcal{A}|$, then $|\mathcal{B}|$ can be bounded using the same method.
    Fix any $B\in\mathcal{B}$. Every $A\in\mathcal{A}$ must satisfy $|A\cap B|\ge t$, hence
    \[\begin{aligned}
    |\mathcal{A}|&\le \bigl|\{A\mid |A\cap B|\ge t, |A| = a \}\bigr|\\
    &= \sum_{i = t}^{a}\binom{b}{i}\binom{n-b}{a-i} \le 2^b\binom{n-b}{a-t}.
    \end{aligned}\]
    The last inequality holds since when $n>3a$ and $i\le a$, $\binom{n}{i}$ is monotone in $i$.
\end{proof}

The following lemma shows that under the shifted assumption, if two disjoint subsets $S_1, S_2 \subseteq [\ell]$ both have large degree, then a contradiction can be derived.

\begin{lem}\label{lem:large degree intersecting}
    Let $\mathcal{F}$ be an $\ell$-shifted $k$-uniform intersecting family on $[n]$ with $n \ge 4k+\ell$, and let $S_1, S_2 \subseteq [\ell]$ be disjoint subsets. Then either
    \[
    d_{S_1}(\mathcal{F}) \le 2^{k-|S_2|}\binom{n-\ell-k+|S_2|}{k - \ell+|S_2|-1}
    \]
    or
    \[
    d_{S_2}(\mathcal{F}) \le 2^{k-|S_1|}\binom{n-\ell-k+|S_1|}{k - \ell+|S_1|-1}.
    \]
\end{lem}
\begin{proof}
    For a set $S\subset [\ell]$, let 
    \[\mathcal{F}(S):=\{A\backslash S|S\subset A\in \mathcal{F}\}.\]
    Then $d_S(\mathcal{F}) = |\mathcal{F}(S)|$.
    By Lemma \ref{lem:intersecting lemma} c), for each set $A\in \mathcal{F}(S_1)$, $B\in \mathcal{F}(S_2)$, we have $|A\cap B|\ge \ell+1 - |S_1| - |S_2|$, therefore $\mathcal{F}(S_1)\subset \binom{[n-\ell]}{k-|S_1|}$ and $\mathcal{F}(S_2)\subset\binom{[n-\ell]}{k-|S_2|}$ are $(\ell+1-|S_1|-|S_2|)$-cross-intersecting.

    If $\mathcal{F}(S_2)$ is empty, then $d_{S_2}(\mathcal{F}) = 0$ so the statement holds.
    If $\mathcal{F}(S_2)$ is non-empty, then by Lemma \ref{lem:cross-intersecting size},
    \[\begin{aligned}|\mathcal{F}(S_1)|&\le 2^{k-|S_2|}\binom{n-\ell-k+|S_2|}{k-|S_1|-(\ell+1 - |S_1| - |S_2|)} \\&= 2^{k-|S_2|}\binom{n-\ell-k+|S_2|}{k-\ell-1+|S_2|}.\end{aligned}\]
    Hence the statement is true.
\end{proof}

\begin{cor}\label{cor:large degree intersecting}
    Let $\mathcal{F}$ be an $\ell$-shifted $k$-uniform intersecting family. Suppose $m = \lfloor\frac{\ell-1}{2}\rfloor$, $n \ge 4k+\ell$, and define $\mathcal{G}$ as the following family of subsets of $[\ell]$: 
    \[\mathcal{G}:=\bigl\{S\subset [\ell] \mid |S|\le m, d_S(\mathcal{F})\ge 2^{k-m}\tbinom{n-k-\ell+m}{k-\ell+m}\bigr\}.\]Then $\mathcal{G}$ is intersecting.
\end{cor}
\begin{proof}
    Let $f(j) = 2^{k-j}\binom{n-k-\ell+j}{k-\ell+j-1}$.
    Note that when $n>3k$, for $j \le m \le \ell$, we have
    \[\begin{aligned}
    f(j+1) & = f(j) \cdot \frac{\binom{n-k-\ell+j+1}{k-\ell+j}}{2 \binom{n-k-\ell+j}{k-\ell+j-1}}=f(j) \cdot \frac{n-k-\ell+j+1}{2(k-\ell+j)} \ge f(j).
    \end{aligned}\]
 If $S_1, S_2\in \mathcal{G}$ are disjoint, then by Lemma \ref{lem:large degree intersecting}, we have either 
    \[
    d_{S_1}(\mathcal{F}) \le 2^{k-|S_2|}\binom{n-\ell-k+|S_2|}{k - \ell+|S_2|-1}
    \]
    or
    \[
    d_{S_2}(\mathcal{F}) \le 2^{k-|S_1|}\binom{n-\ell-k+|S_1|}{k - \ell+|S_1|-1}.
    \]
    Without loss of generality, we assume that $d_{S_1}(\mathcal{F})\le 2^{k -|S_2|}\binom{n-k-\ell+|S_2|}{k-\ell-1+|S_2|}$, then by the monotonicity of $f(j)$, we have
    \[\begin{aligned}
    d_{S_1}(\mathcal{F})&\le 2^{k -|S_2|}\binom{n-k-\ell+|S_2|}{k-\ell-1+|S_2|}=f(|S_2|) \le f(m) \\
    &=2^{k-m} \binom{n-k-\ell+m}{k-\ell+m-1}< 2^{k - m}\binom{n-k-\ell+m}{k-\ell+m}.
    \end{aligned}\]
    This contradicts the assumption $S_1\in \mathcal{G}$.
\end{proof}

\noindent\textbf{Proof of Theorem \ref{thm:k+2 degree}.}\medskip

Let $m = \lfloor (k+1)/2\rfloor$. By Proposition \ref{prop:shifted assumption}, we may assume that $\mathcal{F}$ is $(k+2)$-shifted and that $d_1 \ge d_2 \ge \dots \ge d_{k+2}$ are the $k+2$ largest degrees. For a subset $X\subset [k+2]$, we define the degree of $X$ in $\mathcal{F}$ as 
\[d_X(\mathcal{F}) = \bigl|\{A\in\mathcal{F}\mid A\cap [k+2] =X\}\bigr|.\]
We are given the condition $n > 12k$. Let $\mathcal{G}$ be the following family:
\[\mathcal{G}:=\bigl\{S\subset [k+2] \mid |S|\le m, d_S(\mathcal{F})\ge 2^{k-m}\tbinom{n-2k-2+m}{m-2}\bigr\}.\]
By Corollary \ref{cor:large degree intersecting} with $\ell = k+2$, $\mathcal{G}$ is an intersecting family in $2^{[k+2]}$. Its upward closure $\overline{\mathcal{G}}$ is also intersecting. Therefore Lemma \ref{lem:uniform low degree} shows there exists some element $i \in [k+2]$ such that $d^{(j)}_i(\overline{\mathcal{G}})\le \binom{k}{j-2}$ for $j = 1, 2, \dots, m$, in particular, $d^{(2)}_i(\overline{\mathcal{G}})\in \{0, 1\}$. Below we discuss two cases according to the value of $d_i^{(2)}(\overline{\mathcal{G}})$.\medskip

\noindent\textbf{Case 1:} $d_i^{(2)}(\overline{\mathcal{G}}) = 0$.\smallskip

We split the sets in $\mathcal{F}$ that contain $i$ into the following two parts and estimate the size of each part separately:
\[D_1 = \big\{A\big|i\in A\in\mathcal{F}, A\cap [k+2]\in \mathcal{G}\big\},\]
\[D_2 = \big\{A\big|i\in A\in \mathcal{F}, A\cap[k+2]\not\in\mathcal{G}\big\}.\]
We first estimate $|D_1|$:
\[\begin{aligned}
    |D_1| &= \sum_{i\in X\in\mathcal{G}}d_X(\mathcal{F})
     \le \sum_{j\ge 2}d_i^{(j)}(\mathcal{G})\binom{n-k-2}{k-j}
    \le \sum_{j\ge 2}d_i^{(j)}(\overline{\mathcal{G}})\binom{n-k-2}{k-j}\\
    &\le \sum_{j\ge 3}\binom{k}{j-2}\binom{n-k-2}{k-j} = \binom{n}{k-2} - \binom{n-k-2}{k-2}.
\end{aligned}\]
The first inequality is because for fixed $X\in \mathcal{G}$, the number of $k$-set $A$ with $A\cap [k+2] = X$ is at most $\binom{n-k-2}{k-|X|}$. The second inequality follows from $\mathcal{G}\subset \overline{\mathcal{G}}$. The third inequality follows from Lemma \ref{lem:uniform low degree} and $d^{(2)}_i(\overline{\mathcal{G}}) = 0$.

Next we estimate $|D_2|$. We prove the following claim: for $X\notin \mathcal{G}$, we have $d_{X}(\mathcal{F})\le 2^{k-m}\binom{n-k-2}{k-m}$.

First suppose $|X|\le m$, since $2m\le k+1<k+2$, i.e. $k-m>m-2$, by the definition of $\mathcal{G}$,  we have
\[d_X(\mathcal{F})\le 2^{k-m}\binom{n-2k-2+m}{m-2}\le 2^{k-m}\binom{n-k-2}{k-m}.\]
If $|X|> m$, then
\[d_X(\mathcal{F})\le \binom{n-k-2}{k-m}\le 2^{k-m}\binom{n-k-2}{k-m}.\]

Using the claim, $|D_2|$ can be bounded by
\[\begin{aligned}
    |D_2| &= \sum_{X\not\in\mathcal{G}}d_X(\mathcal{F})
    \le \sum_{X\subset [k+2]} 2^{k-m}\binom{n-k-2}{k-m}
    \le 2^{2k+2-m}\binom{n-k-2}{k-m}.
\end{aligned}\]
Next we show that for sufficiently large $k$, when $n>12k$, we have
\[2^{2k+2-m}\binom{n-k-2}{k-m}\le \binom{n-k-2}{k-2}.\]
Recall that  $k/2\le m\le (k+1)/2$, thus
\[\begin{aligned}
    \dfrac{\binom{n-k-2}{k-2}}{\binom{n-k-2}{k-m}} &= \prod_{j = 1}^{m-2} \dfrac{n-2k+j}{k-m+j}
    \ge \left(\dfrac{n-2k}{k}\right)^{m-2}
    \ge 10^{m-2} \ge 10^{\frac{k}{2}-2}.
\end{aligned}\]
When $k$ is sufficiently large, 
\[\begin{aligned}
    2^{2k+2-m}&\le 2^{\frac{3k}{2}+2} = 4\cdot 8^{\frac{k}{2}}\le 10^{-2}\cdot 10^{\frac{k}{2}}\le\dfrac{\binom{n-k-2}{k-2}}{\binom{n-k-2}{k-m}}.
\end{aligned}\]
This implies that \[d_i = |D_1| + |D_2|\le \binom{n-2}{k-2} - \binom{n-k-2}{k-2}+\binom{n-k-2}{k-2}\le \binom{n-2}{k-2}.\]
which completes the proof of Case 1.\\

\noindent\textbf{Case 2:} $d_i^{(2)}(\overline{\mathcal{G}}) = 1$.\smallskip

Assume, without loss of generality, that $\{1, i\} \in \overline{\mathcal{G}}$. Because $\overline{\mathcal{G}}$ is upward closed, for any $j \le m$ there are exactly $\binom{k}{j-2}$ $j$-element subsets of $[k+2]$ that contain $\{1, i\}$. Since $d_i^{(j)}(\overline{\mathcal{G}}) \le \binom{k}{j-2}$ for $j \le m$ by Lemma \ref{lem:uniform low degree}, every $j$-element set in $\overline{\mathcal{G}}$ containing $i$ must also contain $1$.

Moreover, because $\overline{\mathcal{G}}$ is intersecting and $\{1,i\} \in \overline{\mathcal{G}}$, every set in $\overline{\mathcal{G}}$ must intersect $\{1,i\}$. Combined with the previous observation, it follows that \emph{every} set in $\overline{\mathcal{G}}$ must contain the element $1$.
Therefore, for every $j \neq 1$, we still have $d_j^{(u)}(\overline{\mathcal{G}}) \le \binom{k}{u-2}$ for all $u \le m$ (the same bound as in the Lemma \ref{lem:uniform low degree}).

Now we discuss two subcases.\smallskip

\noindent\textbf{Subcase 2.1:} There exists an element $j \neq 1$ such that $d_j^{(2)}(\overline{\mathcal{G}}) = 0$. In this situation, we can apply exactly the same argument as in Case 1 (with $j$ playing the role of $i$) to conclude that $d_{k+2}(\mathcal{F}) \le \binom{n-2}{k-2}$.\smallskip

\noindent\textbf{Subcase 2.2:} For every $j \neq 1$, we have $d_j^{(2)}(\overline{\mathcal{G}}) = 1$. This means that $\{1, j\} \in \overline{\mathcal{G}}$ for all $j \in [k+2] \setminus \{1\}$. Then either $\{1\} \in \mathcal{G}$ or, for each $j \neq 1$, the set $\{1, j\}$ itself belongs to $\mathcal{G}$. We claim that every set in $\mathcal{F}$ contains the element $1$.

Now suppose, for contradiction, that there exists a set $A \in \mathcal{F}$ that does \emph{not} contain the element $1$. We can pick an element $i \notin A \cap [k+2]$ with $i \neq 1$.
Recall that $\mathcal{F}$ is $(k+2)$-shifted. By repeatedly applying Proposition \ref{prop:set shifting}, we can shift the set $A$ into $[k+2]\setminus\{1, i\}$, and conclude that $A' := [k+2]\setminus\{1, i\} \in \mathcal{F}$. 
Because either $\{1,i\} \in \mathcal{G}$ or $\{1\} \in \mathcal{G}$, there exists a set $B \in \mathcal{F}$ such that $B \cap [k+2] = \{1, i\}\text{ or }\{1\}$. In either case we have $A' \cap B = \emptyset$, a contradiction.
Therefore, every set in $\mathcal{F}$ must contain the element $1$. Consequently,
\[
d_{k+2}(\mathcal{F}) \le \binom{n-2}{k-2}.
\]

This completes the proof for Case 2, and thus the proof of Theorem \ref{thm:k+2 degree}. \qed

\medskip

Using similar techniques, we proof Theorem \ref{thm:l+1 degree}.\medskip

\noindent\textbf{Proof of Theorem \ref{thm:l+1 degree}.\medskip}

Let $m = \lfloor\frac{\ell}{2}\rfloor$. By Proposition \ref{prop:shifted assumption}, we can assume that $\mathcal{F}$ is $(\ell+1)$-shifted. Throughout the proof, we define the degree $d_X$ of a subset $X\subset[\ell+1]$(with respect to $\mathcal{F}$) by
\[d_X(\mathcal{F}):=\bigl|\{A\in\mathcal{F}\mid A\cap [\ell+1] = X\}\bigr|.\]
Denote by $\mathcal{G}$ the following set family:
\[\mathcal{G}:=\big\{X\subset [\ell+1]\big||X|\le m, d_X(\mathcal{F})\ge 2^k\binom{n-k}{k-m}\big\}\]
Since $2^k>2^{k-m}, k-m\ge k-\ell-1+m$, and $n-k\ge n-k-\ell-1+m$, when $n>5k$ the monotonicity of binomial coefficients yields 
\[2^k\binom{n-k}{k-m}\ge 2^{k-m}\binom{n-k-\ell-1+m}{k-\ell-1+m},\] so by Corollary \ref{cor:large degree intersecting}, $\mathcal{G}$ is an intersecting family on $2^{[\ell+1]}$, then by Lemma \ref{lem:uniform low degree}, there is an $i\in [\ell+1]$ such that $d_i^{(j)}(\overline{\mathcal{G}})\le \binom{\ell-1}{j-2}$ for $1\le j\le m$.
If $d_i^{(2)}(\overline{\mathcal{G}}) = 1$, without loss of generality, we may assume that $\{1, i\}\in\overline{\mathcal{G}}$. Next we consider two cases according to the value of $d_{\{1, i\}}(\overline{\mathcal{G}})$.

\noindent\textbf{Case 1:} $d_i^{(2)}(\overline{\mathcal{G}}) = 0$ or $d_{\{1, i\}}(\mathcal{F}) \le \binom{n-\ell-1}{k-2} - \binom{n-k-2}{k-2}$.\smallskip

We split the set in $\mathcal{F}$ that contain $i$ into the following two parts and estimate the size of each part separately:
\[D_1 = \{A|i\in A\in \mathcal{F}, A\cap [\ell+1]\in\mathcal{G}\}\]
\[D_2 = \{A|i\in A\in\mathcal{F}, A\cap [\ell+1]\not\in \mathcal{G}\}\]

$|D_1|$ can be bounded using the inequalities $d_i^{(j)}(\overline{\mathcal{G}})\le \binom{\ell-1}{j-2}$:
\[\begin{aligned}
    |D_1| &= \sum_{i\in X\in\mathcal{G}} d_X(\mathcal{F})\\
    & \le \sum_{j\ge3}d_i^{(j)}(\mathcal{G})\binom{n-\ell-1}{k-j} + d_{\{1, i\}}(\mathcal{F})\\
     & \le \sum_{j\ge3}d_i^{(j)}(\overline{\mathcal{G}})\binom{n-\ell-1}{k-j} + d_{\{1, i\}}(\mathcal{F})\\
    &\le \sum_{j \ge 3}\binom{\ell-1}{j-2}\binom{n-\ell-1}{k-j}+\binom{n-\ell-1}{k-2}-\binom{n-k-2}{k-2}\\
    & = \binom{n-2}{k-2} - \binom{n-k-2}{k-2}
\end{aligned}\]

Next we estimate $|D_2|$. We claim that for $X\not\in \mathcal{G}$, we have $d_X(\mathcal{F})\le 2^{k}\binom{n-k}{k-m}$. 

If $|X|\le m$, then the inequality holds follows from the definition of $\mathcal{G}$. If $|X|>m$, then
\[\begin{aligned}
    d_X(\mathcal{F})&\le \binom{n-\ell-1}{k-m}
    \le \sum_{j= 0}^{k-m}\binom{n-k}{j}\binom{k-\ell-1}{k-m-j}
    \le 2^{k-\ell-1}\binom{n-k}{k-m}
    \le 2^k\binom{n-k}{k-m}
\end{aligned}\]
Hence $|D_2|$ can be bounded by the following inequality,
\[\begin{aligned}
    |D_2|&\le \sum_{i\in X\not\in \mathcal{G}}d_X(\mathcal{F})
     \le 2^{\ell+1}\cdot 2^{k}\binom{n-k}{k-m}. 
\end{aligned}\]
We want to show that, there exists some constant $C_\varepsilon$, such that when $\ell \ge \varepsilon k$ and $n\ge C_{\varepsilon}k$, 
\[2^{\ell+k+1}\binom{n-k}{k-m}\le \binom{n-k-2}{k-2}.\]
In fact, we can derive the following:
\[\begin{aligned}
    \dfrac{\binom{n-k-2}{k-2}}{\binom{n-k}{k-m}}&\ge \left(\dfrac{n-2k}{k}\right)^{m-2}\cdot \left(\dfrac{n-2k}{n-k}\right)^2
    \ge 2^{-2}\left(\dfrac{n}{k}-2\right)^{\frac{\ell}{2}-3}.
\end{aligned}\]
When $C_{\varepsilon} \ge 3^{2+\frac{2}{\varepsilon}}>3^{2 +\frac{2k}{\ell}}$, for sufficiently large $k$, when $n\ge C_{\varepsilon}k$,
\[\begin{aligned}
    \dfrac{\binom{n-k-2}{k-2}}{\binom{n-k}{k-m}}
    \ge 2^{-2}\left(\dfrac{n}{k}-2\right)^{\frac{\ell}{2}-3}
    \ge 2^{-2} \left(3^{2+\frac{2k}{\ell}}\right)^{\frac{\ell}{2}-3}
    \ge 2^{k+\ell+1}.
\end{aligned}\]
Where the last inequality holds when $k$ is sufficiently large.

Therefore, taking $C_\varepsilon = 3^{2+\frac{2}{\varepsilon}}$, in Case 1 we have \[d_i(\mathcal{F})\le \binom{n-2}{k-2}\le \binom{n-2}{k-2} + \binom{n-\ell-1}{k-\ell},\]
which is what we wanted.\\

\noindent\textbf{Case 2:} $d_i^{(2)}(\overline{\mathcal{G}}) = 1$ and $d_{\{1, i\}}(\mathcal{F}) > \binom{n-\ell-1}{k-2} - \binom{n-k-2}{k-2}$.
Then, by the same argument as in the proof of Theorem \ref{thm:k+2 degree}, we have $d_{j}^{(u)}(\overline{\mathcal{G}}) \le \binom{\ell-1}{u-2}$ for all $j \neq 1$ and $u \le m$. Consequently, if for some $j \neq 1$ we have $d_{\{1, j\}}(\mathcal{F}) \le \binom{n-\ell-1}{k-2} - \binom{n-k-2}{k-2}$, then the proof is already complete (by reducing to the situation of Case 1). 
Thus we may assume that for every $j \neq 1$,
\[
d_{\{1, j\}}(\mathcal{F}) > \binom{n-\ell-1}{k-2} - \binom{n-k-2}{k-2}.
\]
We shall show that under this assumption every set in $\mathcal{F}$ must either contain the element $1$ or contain the whole set $\{2,3,\dots ,\ell+1\}$.

Assume, for contradiction, that there exists a set $A \in \mathcal{F}$ that not contains $1$ and $i$ with $2 \le i \le \ell+1$. By part (a) of Lemma \ref{lem:intersecting lemma} (upward‑closedness), we can find a set $B \in \mathcal{F}$ such that $B \cap [\ell+1] = [\ell+1] \setminus \{1, i\}$. Then $|B \setminus [\ell+1]| = k - \ell + 1$.

Now consider the family
\[
\mathcal{F}(\{1, i\}) := \{ A \setminus \{1, i\} \mid A\cap [\ell+1] = \{1, i\}, A \in \mathcal{F} \}.
\]
Every member of $\mathcal{F}(\{1, i\})$ must intersect $B \setminus [\ell+1]$ (otherwise the corresponding original set $A$ would be disjoint from $B$). Hence
\[
|\mathcal{F}(\{1, i\})| \le \binom{n-\ell-1}{k-2} - \binom{n-k-2}{k-2}.
\]
But $|\mathcal{F}(\{1, i\})| = d_{\{1, i\}}(\mathcal{F})$, which contradicts our assumption that $d_{\{1, i\}}(\mathcal{F}) > \binom{n-\ell-1}{k-2} - \binom{n-k-2}{k-2}$. Therefore, in this case every set in $\mathcal{F}$ either contains $1$ or contains $\{2,3,\dots ,\ell+1\}$. Consequently,
\[
d_{\ell+1}(\mathcal{F}) \le \binom{n-2}{k-2} + \binom{n-\ell-1}{k-\ell}.
\] \qed

\section{Concluding Remarks}\label{sec:conclusion}

In this paper, we study the problem of determining the maximum value of the $(\ell+1)$-th largest degree of an intersecting family $\mathcal{F} \subset \binom{[n]}{k}$. For $\ell=2k$, Theorem \ref{thm:2k+1 degree} provides the precise range that the $1$-star gives the maximum, resolving the Frankl--Wang Conjecture. For $\ell=k+1$, Theorem \ref{thm:k+2 degree} shows that the $1$-star example gives the maximum for $n>12k$. A natural question is whether the constant $12$ can be reduced.

\begin{que}\label{que:2k+3 conjecture}
    Can we find a constant $M$ such that for every $n \ge 2k + M$, every intersecting family $\mathcal{F} \subset \binom{[n]}{k}$ satisfies
    \[
    d_{k+2}(\mathcal{F}) \le \binom{n-2}{k-2}?
    \]
    In particular, is $M = 3$ already sufficient?
\end{que}

% \section{concluding remarks}

% In this paper, we mainly consider the $(k+2)$th largest degree and the $(2k+1)$th degree for the intersecting set family. We also give an approach to bound the $l+1$ largest degree for $l\le k$ when $n$ is large. Here are some intriguing problem about the degree Erd\H{o}s-Ko-Rado problem.

% \begin{que}\label{2k+3 conjecture}
%     Are there some constant $M$, such that, for $n \ge 2k+M$, $d_{k+2}\le \Comb{n-2}{k-2}$. Furthermore, does $M=3$ big enough?
% \end{que}

As a remark to Question \ref{que:2k+3 conjecture}, we present the following examples for $n = 2k+2$ and $k \in \{4,5,7\}$. In all these examples, there exists an intersecting family whose $(k+2)$-th degree is strictly greater than $\binom{n-2}{k-2}$.

\begin{exam}\label{exam:counter 2k+2}
    Let $k = 4$, $n = 10$. Construct $\mathcal{F}$ as the union of the following two families:
    \begin{enumerate}[label=(\alph*)]
        \item all $4$-element subsets of $[6]$, i.e., $\binom{[6]}{4}$;
        \item all sets of the form $A \cup \{b\}$, where $A$ is a block of a $2\text{-}(6,3,2)$ design and $b \in \{7,8,9,10\}$.
    \end{enumerate}
    One can verify that $\mathcal{F}$ is intersecting and that the $6$th largest degree in $\mathcal{F}$ equals $30$, which exceeds $\binom{n-2}{k-2} = \binom{8}{2} = 28$.
\end{exam}

\begin{exam}\label{exam:counter k5}
    For $k = 5$ and $n = 12$, take a Fano plane on $[7]$, denoted by $\mathcal{P}$ (a Steiner triple system $S(2,3,7)$). Let $\mathcal{F}$ consist of all $5$-element subsets of $[12]$ that contain at least one triple from $\mathcal{P}$. It can be checked that $\mathcal{F}$ is intersecting under this construction, and that $d_{7}(\mathcal{F}) = 125 > \binom{10}{3} = 120$.
\end{exam}

\begin{exam}\label{exam:counter k6}

For $k=7$ and $n=16$, we let the ground set be $[16]$ and let $\mathcal{F}$ consists of all the $7$-sets whose intersection with $[9]$ contains a subset from the family $\mathcal{G}=\{S \cup \{i\}: S \in \mathcal{H}, i \in \{7,8,9\}\} \cup \{\{j,7,8,9\}: j \in [6]\}$.
Here $\mathcal{H}$ is the unique $2$-$(6,3,2)$ desing on $[6]$. Counting gives  $d_9(\mathcal{F})=2023>2002=\binom{n-2}{k-2}$.
\end{exam}

Additionally, when $n = 2k+1$ for an odd integer $k \ge 3$, we may take $\mathcal{F}$ to consist of all $k$-sets that intersect $[k+2]$ in at least $\frac{k+3}{2}$ elements. It is straightforward to check that  $d_{k+2}(\mathcal{F}) > \binom{n-2}{k-2}$. We omit the calculations here.

% As a remark to the Question \ref{2k+3 conjecture}, we have the following counter example for $n = 2k+2$:
% \begin{exam}
%     When $k = 4$, $n = 10$, we take all the following sets to form $\mathcal{F}$:

%     a) $\Comb{[6]}{4}$

%     b) $A\cup \{b\}$, here $A$ belongs to a $2-(6, 3, 2)$ design \cite{todo}, $b\in \{7, 8, 9, 10\}$.

%     We can verify that $\mathcal{F}$ is intersecting, and the 6th largest degree in $\mathcal{F}$ is 30, larger than $\Comb{n-2}{k-2} = \Comb{8}{2} = 28$.
% \end{exam}
% \begin{rmk}
%     When $k = 5$, $n = 12$, we can also find a counter example in the following way: first we choose a Fano plane on [7], says $\mathcal{P}$, then choose all $5$-sets contains some triple in $\mathcal{P}$, we can verify that $\mathcal{F}$ is intersecting under the construction, and $d_{7}(\mathcal{F}) = 125 > \Comb{10}{3}$.
% \end{rmk}

Regarding the $(\ell+1)$-th largest degree with $\ell \le k$, our Theorem \ref{thm:l+1 degree}  shows that when $\ell = \Theta(k)$, the bound $d_{\ell+1} \le \binom{n-2}{k-2} + \binom{n-\ell-1}{k-\ell}$ holds for $n > \Omega(k)$. Since Frankl and Wang showed that the same bound is valid for $\ell = \Omega(1)$, a natural question is whether a unified statement can be made for all $\ell$.

\begin{que}\label{que:general l}
    Can we find an absolute constant $C$ such that for every $4 \le \ell \le k$ and every $n > Ck$, any intersecting family $\mathcal{F} \subset \binom{[n]}{k}$ (with degrees ordered as $d_1 \ge d_2 \ge \cdots \ge d_n$) satisfies
    \[
    d_{\ell+1}(\mathcal{F}) \le \binom{n-2}{k-2} + \binom{n-\ell-1}{k-\ell}?
    \]
    If such a constant exists, what is the smallest possible value of $C$?
\end{que}
\bigskip
\noindent {\bf Acknowledgment.} The authors thank Yongtao Li for pointing out a misstated theorem used in the previous version of this paper.
% 参考文献
\bibliographystyle{amsplain}
\bibliography{citations}  
% \begin{thebibliography}{99}
% \bibitem{key1} 作者. 标题. 期刊, 年份.
% \bibitem{key2} 作者. 标题. 会议, 年份.
% \bibitem{key3} 作者. 书名. 出版社, 年份.
% \end{thebibliography}

% 附录（可选）
\appendix
% \section{附录标题}
\label{app:technical}

\end{document}